\documentclass[12pt, a4paper]{amsart}
\usepackage{amsfonts, amssymb, amsmath}
\usepackage{amsthm}
\usepackage{tabularx}
\usepackage{enumerate}
\usepackage[all]{xy}

\usepackage{mathtools}
\usepackage{hyperref}

\usepackage{xcolor}

\usepackage{mathtools}
\newcolumntype{C}[1]{>{\centering\arraybackslash}p{#1}}

\newcommand{\Q}{{\mathbb Q}}

\newcommand{\Z}{{\mathbb Z}}

\newcommand{\Oo}{\mathcal{O}}

\newcommand{\B}{\mathcal{B}}

\newcommand{\mmul}{\mbox{mult}}
\newcommand{\tl}[1]{\tilde{#1}}

\newcommand{\st}{^{\ast}}
\newcommand{\ts}{_{\ast}}

\newcommand{\vc}[1]{\vcenter{\hbox{#1}}}

\newtheorem{thm}{Theorem}[section]
\newtheorem{pro}[thm]{Proposition}
\newtheorem{cor}[thm]{Corollary}
\newtheorem{lem}[thm]{Lemma}
\newtheorem{que}[thm]{Question}
\theoremstyle{definition}
\newtheorem{rk}[thm]{Remark}

\newtheorem{cov}[thm]{Convention}
\newtheorem{defn}[thm]{Definition}
\newtheorem*{cl}{Claim}

\begin{document}
\title{Chern numbers of terminal threefolds}
\author{Paolo Cascini}
\address{Department of Mathematics\\
Imperial College London\\
180 Queen's Gate\\
London SW7 2AZ, UK}
\email{p.cascini@imperial.ac.uk}
\author{Hsin-Ku Chen}
\address{ School of Mathematics\\
Korea Institute for Advanced Study\\
 85 Hoegiro, Dongdaemungu \\
 Seoul 02455, Republic of Korea}
\email{hkchen@kias.re.kr}

\begin{abstract} 
Let $X$ be a smooth complex projective threefold. We show that if $X_i\dashrightarrow X_{i+1}$ is a flip which appears in the $K_X$-MMP, then $c_1(X_i)^3-c_1(X_{i+1})^3$ is bounded by a constant depending only on $b_2(X)$.
\end{abstract}
\maketitle

\section{Introduction}

The aim of this paper is to discuss the following question:

\begin{que}[Kotschich \cite{k1}]
	Let $X$ be a smooth complex projective threefold. Are Chern numbers of $X$ bounded by a number that depends only on the topology of the manifold underlying $X$?
\end{que}
Note that this question is known to have a negative answer for non-K\"ahler complex threefolds \cite{lb} and  for complex projective varieties of dimension greater than three \cite{st1}. On the other hand, the question has a positive answer in the case of K\"ahler varieties underlying a spin manifold \cite{st3}.

\medskip

In the case of a smooth projective threefold $X$, the only Chern numbers are $c_1^3(X)$, $c_1c_2(X)$ and $c_3(X)$. The last one coincides with the topological Euler characteristic of $X$ and, in particular, it is a topological invariant. On the other hand, by Hirzebruch-Riemann-Roch theorem, we have that $|c_1c_2(X)|$ coincides with $|24\chi(\Oo_X)|$ and it is therefore bounded by an integer depending only on the sum of the Betti numbers of $X$. Therefore, it remains to bound $c_1^3(X)$ or, in other words, $K_X^3$.

Thanks to the Minimal Model Program, we know that a smooth projective threefold is birational to either a minimal model, i.e. a variety $Y$ such that $K_Y$ is nef or to a variety $Y$ which admits a Mori fibre space, i.e. a fibration $\eta\colon Y\to Z$ such that $\dim Z<3$, the general fibre of $\eta$ is Fano and the relative Picard number $\rho(Y/Z)$ is one. We first want to bound $K_Y^3$. To this end, if $Y$ is a minimal model then $K_Y^3$ concides with the volume of $Y$ and by \cite{ct} this number is bounded by the Betti numbers of $X$. In the second case, instead, it follows from \cite{st2} that if the cubic form $F_Y$ has non-zero discriminant $\Delta_{F_Y}$ (cf. \S \ref{c_pre}), then $K_Y^3$ is bounded by a number that depends only on $F_Y$ and the first Pontryagin class $p_1(Y)$ of $Y$ (cf. \S \ref{sp}).

We then need to show that these bounds hold also on $X$. By \cite{cz}, we have that the number $k$ of steps of an MMP 
 \[ X=X_0\dashrightarrow X_1\dashrightarrow ...\dashrightarrow X_k=Y\] 
 starting from $X$ and the singularities of the output $Y$ of this MMP are both bounded by a number which depends only on the topology of the manifold underlying $X$. Thus, the primary challenge lies in bounding the variation of the topological invariants of the underlying varieties and the Chern number $K_{X_i}^3$ at each step $X_i\dashrightarrow X_{i+1}$ of this MMP. In the case of divisorial contractions, this problem was solved in \cite{ct}, provided that the discriminant of the associated cubic form $F_{X_i}$ is non-zero. In this paper, we solve the case of flips:

\begin{thm}\label{thm1}
	Let $X$ be a smooth complex projective threefold and  let \[ X=X_0\dashrightarrow X_1\dashrightarrow ...\dashrightarrow X_k=Y\] be a $K_X$-MMP. 
	
	Then $|K_{X_i}^3-K_{X_{i+1}}^3|$ is bounded by an integer which depends only on $b_2(X)$.
\end{thm}

\medskip

The remaining question is to understand how the cubic form varies after each flip. Indeed if $X_i\rightarrow X_{i+1}$ is a divisorial contraction to a curve then  $K_{X_i}^3-K_{X_{i+1}}^3$ 
depends on the cubic form of $F_{X_i}$ of $X_i$ and not only on its Betti numbers ( e.g., consider the blowup of a rational curve of degree $d$ in $\mathbb P^3$). More precisely, 
if  $X_i\dashrightarrow X_{i+1}$ is a flip then 
we need to show:
\begin{enumerate}[(1)]
\item the equivalence class of $F_{X_{i+1}}$ belongs to a finite set which depends only on $F_{X_i}$; and 
\item if $\Delta_{F_{X_i}}\neq 0$ then  $\Delta_{F_{X_{i+1}}}\neq 0$.
\end{enumerate}
We have a partial solution about the finiteness of cubic forms.

\begin{thm}\label{thm2}
	Let $X$ be a smooth complex projective threefold  and let \[ X=X_0\dashrightarrow X_1\dashrightarrow ...\dashrightarrow X_k=Y\] be a $K_X$-MMP. Assume that $X_i\dashrightarrow X_{i+1}$ is a flip for some $i=1,\dots,k-1$, and let $\phi_i\colon X_i\rightarrow W_i$ be the corresponding flipping contraction.
	
	Then the equivalence class of $F_{X_{i+1}}$  belongs to a finite set which depends only on $b_2(X)$, $F_{X_i}$ and $\phi_i\st H^2(W_i,\Z)\subset H^2(X_i,\Z)$
\end{thm}

We would like to thank P. Wilson for pointing out a mistake in the first version of this paper. We are grateful to the referees for carefully reading the paper and for several useful comments. The first author would like to thank the National Center for Theoretical Sciences in Taipei and  J. A. Chen for their generous hospitality, where some of the work for this paper was completed. 
Some of the work was done when the second author was visiting Imperial College London. The second author thanks Imperial College London for their hospitality. 
The first author is partially supported by an EPSRC grant. 
The second author is supported by KIAS individual Grant MG088901.

\section{Preliminaries}

\subsection{Terminal threefolds}
We work over the complex numbers.

\subsubsection{Cubic Forms}\label{c_pre}
Let $R$ be a commutative ring. A \emph{cubic form} in $R$ is an homogenous polynomial $F\in R[x_0,\dots,x_n]$  of degree three. 
We denote by $\Delta_{F}$ its discriminant (e.g. see \cite[\S2.3]{ct}). Recall that if $R$ is an algebraically closed field of characteristic zero then $(F=0)\subset \mathbb P^n$ is singular if and only if $\Delta_F=0$. 
Following \cite[Definition 2.9]{ct}, given a cubic form $G\in  R[x_1,\dots,x_n]$ and  elements $b=(b_1,\dots,b_n)\in R^n$ and $a\in R$, we denote by $(a,b,G)$ the cubic form 
\[
F(x_0,\dots,x_n) \coloneqq ax_0^3+x_0^2\cdot \sum_{i=1}^n b_ix_i+ G(x_1,\dots, x_n)\in  R[x_0,\dots,x_n]
\]
 We say that two cubic forms $F_1,F_2\in R[x_0,\dots,x_n]$ are \emph{equivalent} if there exists $T\in {\rm SL}(n+1,R)$ such that $F_1(T\cdot x)=F_2(x)$. If the cubic form  $F\in R[x_0,\dots,x_n]$ is equivalent to $(a,b,G)$ for some cubic form $G\in R[x_1,\dots,x_n]$ and elements 
 $b=(b_1,\dots,b_n)\in R^n$ and $a\in R$,
 then, for simplicity, we will just denote it by
 \[
 F\sim (a,b,G).
 \]

	Let $X$ be a threefold. Given $\sigma_1$, $\sigma_2\in H^2(X,\Z)$, we denote by  $\sigma_1\cdot \sigma_2\in H_2(X,\Z)$  the image of $\sigma_1\cup\sigma_2$ under the morphism $H^4(X,\Z)\rightarrow H_2(X,\Z)$ defined by taking the cap product with the fundamental class of $X$. 
	We denote by $F_X$ the cubic form defined by the cup product on $H^2(X,\mathbb Z)$.

\subsubsection{The first Pontryagin class }\label{sp}

Let $X$ be a terminal threefold. As in \cite[\S3.1]{st2}, we define the first Pontryagin class $p_1(X)\in {\rm Hom}(H^2(X,\mathbb Q),\mathbb Q)$ of $X$, as 
\[
p_1(X)=c_1(X)^2-2c_2(X).
\]
Note that if $X$ is smooth, then $p_1(X)$ only depends on the topology of the underlying manifold of $X$ (e.g. see \cite[\S 1.1]{ov}).

\subsubsection{Singularities of terminal threefolds} 
	Three dimensional terminal singularities were classified by Reid \cite{r} and Mori \cite{m}. A three-dimensional Gorenstein terminal singularity is an isolated compound Du Val singularity, i.e. an isolated hypersurface singularity defined by the equation 
	\[
	f(x,y,z)+ug(x,y,z,u)=0
	\] 
	such that $(f(x,y,z)=0)\subset \mathbb A^3$ defines a two-dimensional Du Val singularity. A three-dimensional non-Gorenstein terminal singularity belongs to one of the following six classes: $cA/r$, $cAx/2$, $cAx/4$, $cD/2$, $cD/3$ and $cE/2$ (see \cite[(6.1)]{r2} for the explicit description of each class).
	
	Let $(P\in X)$ be a germ of a three-dimensional terminal point. The \emph{Cartier index} of $X$ at $P$ is the smallest positive integer $r$ such that $rK_X$ is Cartier near $P$. In particular, if $D$ is any Weil $\mathbb Q$-Cartier divisor on $X$ then $rD$ is also Cartier near $P$. 
	It is known that one can deform this singularity to get a family of terminal singularities, such that a general member of this family has only cyclic-quotient singularities (cf. \cite[(6.4)]{r2}). Assume that $(P\in X)$ deforms to $k$ cyclic quotient points $P_1$, ..., $P_k$. Since $P_i$ is a three-dimensional terminal cyclic-quotient singularity, it is of  type $\frac{1}{r_i}(1,-1,b_i)$ for some $0<b_i\leq \frac{r_i}{2}$.
\begin{defn}
	Notation as above.
	\begin{enumerate}[(1)]
	\item The data $\B(P\in X)\coloneqq \{(r_i,b_i)\}_{i=1}^k$ is called the \emph{basket data} of the singularity  $(P\in X)$. We denote  $\B(X)\coloneqq \bigcup_{P} \B(P\in X)$.
	\item The number $aw(P\in X)\coloneqq k$ is called the \emph{axial weight} of the singularity.
	\item We denote $\Xi(P\in X)\coloneqq r_1+...+r_k$.
	\end{enumerate}	   
\end{defn}
We refer to \cite[Remark 2.1]{ch} for explicit values of these invariants.
\begin{rk}
	In \cite{km} Koll\'{a}r and Mori define another invariant called the \emph{axial multiplicity} of $(P\in X)$. If $P\in X$ is not a $cAx/4$ singularity, than the axial multiplicity coincides with the axial weight. If $P\in X$ is a $cAx/4$ point then the axial multiplicity is equal to $aw(P\in X)+1$. 
\end{rk} 
\begin{defn}
	Let $Y\rightarrow X$ be a divisorial contraction between terminal varieties, which contracts a divisor $E$ to a point $P\in X$. We say that $Y\rightarrow X$ is a \emph{$w$-morphism} if $a(X,E)=\frac{1}{r_P}$, where $r_P$ is the Cartier index of $K_X$ near $P$ and $a(X,E)$ is the discrepancy of $X$ with respect to $E$.
\end{defn}
\begin{defn}
	The \emph{depth} of a terminal singularity $P\in X$, denoted by $dep(P\in X)$, is the minimal length of a sequence \[ X_m\rightarrow X_{m-1}\rightarrow \cdots\rightarrow X_1\rightarrow X_0=X,\] such that $X_m$ is Gorenstein and $X_i\rightarrow X_{i-1}$ is a $w$-morphism for all $1\leq i\leq m$. 
	Given a terminal threefold $X$, we define \[dep(X)\coloneqq \sum_P dep(P\in X).\]
		Note that the existence of such a sequence follows from \cite[Theorem 1.2]{h}.
\end{defn}
\begin{lem}\label{bka}
	Let $X$ be a terminal threefold and $P\in X$ be a singular point. 
	
	Then the following holds.
	\begin{enumerate}[(1)]
	\item The basket data of $P\in X$ belongs to a finite set which depends only on $dep(P\in X)$.
	\item The Cartier index of $X$ at $P$ and the axial multiplicity of $P\in X$ are both bounded by $2dep(P\in X)$.
	\end{enumerate}
\end{lem}
\begin{proof}
	By \cite[Lemma 3.2]{cz} we know that $\Xi(P\in X)\leq 2dep(X)$. Hence the basket data of $P\in X$ belongs to a finite set which depends only on $dep(P\in X)$ and (1) follows. By \cite[Remark 2.1]{ch} we know that the Cartier index of $K_X$ near $P$ and the axial multiplicity of $P\in X$ are both bounded by $\Xi(P\in X)$. This proves (2).
\end{proof}

\subsubsection{The singular Riemann-Roch formula}\label{srr}
Given a projective terminal threefold $X$ and a Weil divisor $D$ on $X$, we have the following version of the Riemann-Roch formula due to Reid \cite{r2}:
\begin{align*}
\chi(\Oo_X(D))=\chi(\Oo_X)&+\frac{1}{12}D(D-K_X)(2D-K_X)+\frac{1}{12}D.c_2(X)\\&+
	\sum_{P\in\B(X)}\left(-i_P\frac{r_P^2-1}{12r_P}+\sum_{j=1}^{i_P-1}\frac{\overline{jb_P}(r_P-\overline{jb_P})}{2r_P}\right),
\end{align*}
where $\B(X)=\{(r_P,b_P)\}$ is the basket data of $X$ and $i_P$ is an integer such that $\Oo_X(D)\cong\Oo_X(i_PK_X)$ near $P$.\par
In particular, if  $D=K_X$, we have \[K_X \cdot c_2(X)=-24\chi(\Oo_X)+\sum_{P\in\B(X)}\left(r_P-\frac{1}{r_P}\right).
\]
\subsubsection{Chen-Hacon factorisation}\label{sch}
We will extensively use the following result: 

\begin{thm}\cite[Theorem 3.3]{ch} \label{chf}
	Assume that either $X\dashrightarrow X'$ is a flip over $W$, or $X\rightarrow W$ is a divisorial contraction to a curve such that $X$ is not Gorenstein over $W$. 
	
	Then there exists a diagram
	\[\vc{\xymatrix@C=0.8cm{Y_1\ar[d] \ar@{-->}[r] & ... \ar@{-->}[r] & Y_k\ar[d]\\ X\ar[rd] & & X'\ar[ld] \\ & W & }}\]
such that 
\begin{enumerate}
\item $Y_1\rightarrow X$ is a $w$-morphism, 
\item $Y_k\rightarrow X'$ is a divisorial contraction, 
\item $Y_1\dashrightarrow Y_2$ is either a flip or a flop over $W$ and 
\item $Y_i\dashrightarrow Y_{i+1}$ is a flip over $W$ for $i>1$. 
\end{enumerate}
Moreover,  if $X\rightarrow W$ is a divisorial contraction then $X'\rightarrow W$ is a divisorial contraction to a point.
\end{thm}
\begin{rk}\label{flpc}
	Notation as in the above theorem. 
	\begin{enumerate}[(1)]
	\item Assume that $X\dashrightarrow X'$ is a flip, $C_{Y_1}$ is the flipping/flopping curve of $Y_1\dashrightarrow Y_2$ and $C_X$ is the image of $C_{Y_1}$ on $X$. Then $C_X$ is a flipping curve of $X\dashrightarrow X'$. This fact follows from the construction of the diagram.
	\item By \cite[Proposition 3.5]{ch}, we have
	 \[dep(X)-1=dep(Y_1)\geq dep(Y_2)>...>dep(Y_k).\]  In particular, $k\leq dep(X)+1$.
	\item  By \cite[Proposition 3.5 and the proof of Proposition 3.6]{ch}, we have
	\[
	dep(X')<dep(X).
	\]
	\end{enumerate}
\end{rk}

\begin{lem}\label{depr}
	Assume that $X$ is a smooth threefold and \[ X=X_0\dashrightarrow X_1\dashrightarrow ...\dashrightarrow X_k\] is a sequence of steps of a  $K_X$-MMP. 
	
	Then $dep(X_i)<\rho(X)$ for all $i$.
\end{lem}
\begin{proof}
	By \cite[Proposition 2.15, Proposition 3.5, Proposition 3.6]{ch} we know that $dep(X_i)\geq dep(X_{i+1})-1$ if $X_i\rightarrow X_{i+1}$ is a divisorial contraction, and $dep(X_i)>dep(X_{i+1})$ if $X_i\dashrightarrow X_{i+1}$ is a flip. It follows that if $m$ is the number of divisorial contractions in the MMP, then $dep(X_i)\leq m$ for all $i$. Since $m<\rho(X)$, we know that $dep(X_i)<\rho(X)$.
\end{proof}
\subsection{Negativity lemma}

We recall the negativity lemma for flips:
\begin{lem}\label{ntl} Let $(X,D)$ be a log pair and  assume that $X\dashrightarrow X'$ is a $(K_X+D)$-flip. 
	
	Then for any exceptional divisor $E$ over $X$, we have that $a(E,X,D)\leq a(E,X',D_{X'})$ where $D_{X'}$ is the strict transform of $D$ in $X'$. Moreover, the inequality is strict if the centre of $E$ in $X$ is contained in the flipping locus.
\end{lem}
\begin{proof}
	The Lemma is a special case of \cite[Lemma 3.38]{km2}.
\end{proof}
\begin{cor}\label{ntl2}
Let $(X,D)$ be a log pair and assume that $X\dashrightarrow X'$ is a $(K_X+D)$-flip and $C\subset X$ is an irreducible curve which is not a flipping curve. 

Then $(K_X+D)\cdot C\geq(K_{X'}+D_{X'})\cdot C_{X'}$ where $C_{X'}$ and $D_{X'}$ are the strict transform on $X'$ of $C$ and $D$ respectively. Moreover, the inequality is strict if $C$ intersects the flipping locus non-trivially.
\end{cor}
\begin{proof}
	Let $\xymatrix{X & W\ar[l]_{\phi} \ar[r]^{\phi'} & X'}$ be a common resolution such that $C$ is not contained in the indeterminacy locus of $\phi$. Then Lemma \ref{ntl} implies that $F\coloneqq \phi\st(K_X+D)-{\phi'}\st(K_{X'}+D_{X'})$ is an effective divisor and is supported on exactly those exceptional divisors whose centres on $X$ are contained in the flipping locus. Hence
	\begin{align*}
		(K_X+D)\cdot C-(K_{X'}+D_{X'})\cdot C_{X'}&=(\phi\st(K_X+D)-{\phi'}\st(K_{X'}+D_{X'}))\cdot C_W\\
		&=F\cdot C_W\geq 0
	\end{align*}
	where $C_W$ is the strict transform of $C$ on $W$. The last inequality is strict if and only if $C_W$ intersects $F$ non-trivially, or equivalently, $C$ intersects the flipping locus	non-trivially.
\end{proof}

\section{Geometry of flips}

\begin{cov}
	Let $H_0$ be a (germ of a) Du Val surface and let $\bar{H}\rightarrow H_0$ be the minimal resolution of $H_0$. We denote $\bar{\rho}(H_0)\coloneqq \rho(\bar{H}/H_0)$.\par
	Let $\phi\colon H\rightarrow H_0$ be a partial resolution of $H_0$, i.e.  $H$  also admits Du Val singularities and $\bar{H}$ is also the minimal resolution of $H_0$. Let $D$ be an effective divisor on $H$. We denote \[\delta_{H\rightarrow H_0}(D)\coloneqq \sum_{i=1}^m\mmul_{\Gamma_i}D\] where the sum runs over all the $\phi$-exceptional divisors.
\end{cov}

\begin{lem}\label{del} 
Fix a positive integer $n$. 
	Let
	\[
	\vc{\xymatrix@R=0.5cm{ X\ar[rd]_{\phi} & & X'\ar[ld]^{\phi'} \\ & W & }}
	\]
	 be a three-dimensional terminal flip and let  $W\rightarrow W_0$ be a birational morphism. Assume that \begin{enumerate}[(i)]
	\item there exist an analytic neighborhood $U_0\subset W_0$ and a Du Val section $H_0\in|-K_{U_0}|$, such that the image of the flipping locus of $X\dashrightarrow X'$ on $W_0$ is contained in $U_0$, and both $H\rightarrow H_0$ and $H'\rightarrow H_0$ are partial resolutions, where $H$ and $H'$ are the proper transforms of $H_0$ on $X$ and $X'$ respectively;
	\item there is an effective divisor $D'\subset X'$ which is contracted by the induced morphism $X'\rightarrow W_0$; and 
	\item $\bar{\rho}(H_0)$, $\delta_{H'\rightarrow H_0}(D'|_{H'})$ and $dep(X)$ are all bounded by $n$.
	\end{enumerate}
	
	Then
	\begin{enumerate}[(1)]
	\item for any flipped curve $C'$, we have that $|D'\cdot C'|$ is bounded by an integer which depends only on $n$; 
	\item  for any flipping curve $C$, if $D$ is the proper transform of $D'$ on $X$ then  $|D\cdot C|$ is bounded by an integer which depends only on $n$; and 
	\item $\delta_{H\rightarrow H_0}(D|_H)$ is bounded by an integer which depends only on $n$.
	\end{enumerate}
\end{lem}
\begin{proof}
Since $H_0$ is Du Val, inversion of adjunction implies  that $(U_0,H_0)$ is canonical. Since  $H\rightarrow H_0$ is a partial resolution, it follows that if $U_X\subset X$ is the pre-image of $U_0$ on $X$ and 
$\psi\colon U_X \to U_0$ is the induced morphism, then $K_{U_X}+H=\psi\st(K_{U_0}+H_{0})$ and, in particular, $H\in |-K_{U_X}|$. Moreover,  all the  flipping curves of $X\dashrightarrow X'$ are contained in $U_X$. Thus, $H\cdot C=-K_X\cdot C$ for any flipping curve $C$. Likewise, $H'\cdot C'=-K_{X'}\cdot C'$ for any flipped curve $C'$.\par 
	Let $C'\subset X'$ be a flipped curve. Since $H'\cdot C'=-K_{X'}\cdot C'<0$, we know that $C'=\Gamma'_k$ for some $1\leq k\leq q$ where $\Gamma'_1,\dots,\Gamma'_q$ are all the exceptional divisors of $H'\rightarrow H_0$. Since $\rho(\bar{H}/H_0)$ is bounded by $n$, the type of singularity of $H_0$ is bounded, so the intersection numbers $\Gamma'_i\cdot \Gamma'_j$ are all bounded by some constant which depends only on $n$ for all $1\leq i,j\leq q$. Thus, $|D'\cdot C'|=|(D'|_{H'}).\Gamma'_k|$ is bounded.\par
	We can write $D'\equiv_W\lambda K_{X'}$ for some rational number $\lambda$.  Notice that $dep(X')<dep(X)\leq n$ and so the Cartier indices of $K_X$ and $K_{X'}$ are both bounded by $2n$ by Lemma \ref{bka}.	Since $|D'\cdot C'|$ is bounded and the Cartier index of $K_{X'}$ is bounded, it follows that  $|\lambda|$ is bounded by an integer which depends only on $n$. Moreover, we have that $D\equiv_W\lambda K_X$. Since $|K_X\cdot C|<1$ by \cite[Theorem 0]{b} and $\lambda$ is bounded, it follows that $|D\cdot C|$ is bounded by an integer that depends only on $n$.\par
	Now we prove (3). Assume first that there is exactly one flipping curve $C\subset X$. Then we can write $D|_H=\sum a_i\Gamma_i+mC$, where $a_i$ and $m$ are a non-negative rational numbers and the sum runs over all the exceptional curves of $H\to H_0$ which are strict transforms on $H$ of those curves $\Gamma'_1,\dots,\Gamma'_q$ in $H'$ which are not flopped curves. We know that $\delta_{H\rightarrow H_0}(D|_H)=\sum_ia_i+m$ and $\sum_ia_i\leq\delta_{H'\rightarrow H_0}(D'|_{H'})$. Since the intersection number \[ D\cdot C=D|_H\cdot C=\sum a_i\Gamma_i\cdot C+mC^2,\] $\Gamma_i\cdot C$ and $C^2$ are all bounded because $\bar{\rho}(H_0)$ is bounded, we know that $m$ is bounded. Thus, $\delta_{H\rightarrow H_0}(D|_H)$ is bounded.\par
	In general we can run analytic $K_X$-MMP over $W$ and get a composition of flips $X=X_0\dashrightarrow X_1\dashrightarrow...\dashrightarrow X_k=X'$ such that the flipping locus of $X_i\dashrightarrow X_{i+1}$ is irreducible for all $i$. One can show that $\delta_{H\rightarrow H_0}(D|_H)$ is bounded by applying the above argument $k$ times.
\end{proof}

\begin{lem}\label{ecb}
Fix a positive integer $n$. 
	Let $Y\rightarrow X$ be a three-dimensional terminal divisorial contraction which contracts a divisor $E$ to a smooth curve $C$.  Assume that 
	\begin{enumerate}[(1)]
	\item $dep(Y)\leq n$;
	\item there exists a birational morphism $X\rightarrow W_0$ such that $C$ is contracted by this morphism;  
	\item there exist an analytic neighbourhood $U_0\subset W_0$ which contains the image of $C$ on $W_0$, and a Du Val section $H_0\in|-K_{U_0}|$, such that $\bar{\rho}(H_0)\leq n$; and 
	\item if  $H_X$ and $H_Y$ are the strict transforms of $H_0$ on $X$ and $Y$ respectively, then 
	$H_Y\rightarrow H_X\rightarrow H_0$ are partial resolutions and we have that $C\subset H_X$. 
	\end{enumerate}
	
	Then $\delta_{H_Y\rightarrow H_0}(E|_{H_Y})$ is bounded by an integer which depends only on $n$.
\end{lem}

\begin{proof}
	We will prove the statement by induction on $dep(Y)$. Assume first that $dep(Y)=0$. In this case $Y$ is Gorenstein and $X$ is smooth near $C$ by \cite[Theorem 4]{cu}. Fix $P\in C$. If $H_X$ is smooth at $P$, then $H_Y\cong H_X$ near $P$. If $H_X$ is singular near $P$. Then the neighbourhood $(P\in C\subset H_X)$ is given by \cite[Theorem 1.1]{j}. Since $\bar{\rho}(H_X)\leq\bar{\rho}(H_0)\leq n$, there are only finitely many possibilities and, therefore, $\delta_{H_Y\rightarrow H_0}(E|_{H_Y})$ is bounded by an integer which depends only on $n$.\par
	In general we have a factorisation
	\[\vc{\xymatrix@C=0.8cm{Z_1\ar[d] \ar@{-->}[r] & ... \ar@{-->}[r] & Z_k\ar[d]\\ Y\ar[rd] & & Y'\ar[ld] \\ & X & }}\] as in Theorem \ref{chf}. We know that $Y'\rightarrow X$ is a divisorial contraction to a point and $Z_k\rightarrow Y'$ is a divisorial contraction which contracts $E_{Z_k}$ to $C'$, where $E_{Z_k}$ is the strict transforms of $E$ on $Z_k$ and $C'$ is the strict transform of $C$ on $Y$.
	\begin{cl}   if $E_{Z_i}$ and $H_{Z_i}$ are the strict transforms of $E$ and $H_X$ on $Z_i$ respectively, then for all $i=1,\dots,k$ we have that 
		$\delta_{H_{Z_i}\rightarrow H_0}(E_{Z_i}|_{H_{Z_i}})$ is bounded by an integer depending only on $n$.
	\end{cl}
	Assuming the claim,  since $\delta_{H_{Z_1}\rightarrow H_0}(E_{Z_1}|_{H_{Z_1}})$ is bounded by an integer depending only on $n$, it follows that $\delta_{H_Y\rightarrow H_0}(E|_{H_Y})$ is bounded by an integer depending only on $n$.
	\medskip
	
	We now prove the claim in several steps: 
	\begin{description}
	\item[Step 1] if $i>1$, or $i=1$ and $Z_1\dashrightarrow Z_2$ is a flip, then $\delta_{H_{Z_i}\rightarrow H_0}(E_{Z_i}|_{H_{Z_i}})$ is bounded by a number which depends only on $n$.
	\end{description}
	Indeed, by \cite[Lemma 3.4]{c} we know that $H_{Z_i}\rightarrow H_X$ is a partial resolution, hence $H_{Z_i}\rightarrow H_0$ is a partial resolution. It also follows by our assumptions and by Remark \ref{flpc} that $dep(Z_i)<dep(Y)\leq n$. Thus, $\delta_{H_{Z_k}\rightarrow H_0}(E_{Z_k}|_{H_{Z_k}})$ is bounded by the induction hypothesis. Now assume that $\delta_{H_{Z_{i+1}}\rightarrow H_0}(E_{Z_{i+1}}|_{H_{Z_{i+1}}})$ is bounded by an integer depending only on $n$, we want to show that $\delta_{H_{Z_i}\rightarrow H_0}(E_{Z_i}|_{H_{Z_i}})$ is also bounded. We know that $dep(Z_i)<n$, hence $\delta_{H_{Z_i}\rightarrow H_0}(E_{Z_i}|_{H_{Z_i}})$ is bounded by Lemma \ref{del}.
	\begin{description}
	\item[Step 2] if $Z_1\dashrightarrow Z_2$ is a flop and no flopped curve is  contained in $H_{Z_2}$,
	 then  $\delta_{H_{Z_1}\rightarrow H_0}(E_{Z_1}|_{H_{Z_1}})=\delta_{H_{Z_2}\rightarrow H_0}(E_{Z_2}|_{H_{Z_2}})$. 
	\end{description}
	Indeed, if $H_{Z_2}$ does not intersect the flopped curve of $Z_1\dashrightarrow Z_2$, then $\delta_{H_{Z_1}\rightarrow H_0}(E_{Z_1}|_{H_{Z_1}})=\delta_{H_{Z_2}\rightarrow H_0}(E_{Z_2}|_{H_{Z_2}})$. If $H_{Z_2}$ intersects the flopped curve, then since $H_{Z_2}\equiv_X-K_{Z_2}$, we know that $H_{Z_2}$ intersects the flopped curve trivially, so $H_{Z_2}$ contains the flopping curve.
	\begin{description}
	\item[Step 3] Let $F=exc(Y'\rightarrow X)$, then $\delta_{H_{Y'}\rightarrow H_0}(F|_{H_{Y'}})$ is bounded by a number which depends only on $n$.
	\end{description}
	Indeed, let $\Xi_1$, ..., $\Xi_m$ be the irreducible components of $F\cap H_{Y'}$, then for all $2\leq j\leq m$ we have that $\Xi_j\equiv\lambda_j\Xi_1$ for some positive rational number $\lambda_j$, where the numerical equivalence is intended as cycles in $Y'$. Moreover, by interchanging $\Xi_1$ and $\Xi_j$ for some $j$ we may assume that $\lambda_j\geq 1$ for all $j$. We can write
	\[F|_{H_{Y'}}=a_1\Xi_1+...+a_m\Xi_m\equiv(a_1+\lambda_2a_2+...+\lambda_ma_m)\Xi_1.\] Then
	\begin{align*}
	0<(a_1+...+a_m)(-F.\Xi_1)&\le (a_1+\lambda_2a_2+...+\lambda_ma_m)(-F.\Xi_1)\\
	&=-F^2.H_{Y'}\\
	&=F^2.K_{Y'}=a(F,X)F^3\leq 4
	\end{align*} where the last inequality follows from \cite[Table 1, Table 2]{k}. By Remark  \ref{flpc}, we have that $dep(Y')\leq n$. Notice that the Cartier index of $F$ divides the Cartier index of $K_{Y'}$ (\cite[Lemma 5.1]{ka}). Thus, Lemma \ref{bka} implies that the Cartier index of $F$ is bounded by $2n$. It follows that 
	\[\delta_{H_{Y'}\rightarrow H_0}(F|_{H_{Y'}})=a_1+...+a_m\leq 8n.\]
	\begin{description}	
	\item[Step 4] Let $F_{Z_j}$ be the strict transform of $F$ on $Z_j$ for $j=1,\dots,k$. Then $\delta_{H_{Z_j}\rightarrow H_0}(F_{Z_j}|_{H_{Z_j}})$ is bounded by a number which depends only on $n$  for any $j=2,\dots,k$.
	\end{description}
	 By Step 3, we know that $\delta_{H_{Y'}\rightarrow H_0}(F|_{H_{Y'}})$ is bounded. Since $\bar{\rho}(H_0)\leq n$, we have that both the singularities of $H_{Y'}$ and the birational morphism $\phi\colon H_{Z_k}\rightarrow H_{Y'}$ have only finitely many possibilities. We know that $F|_{H_{Y'}}$ is supported on the exceptional locus of $H_{Y'}\rightarrow H_0$, hence
	since $F_{Z_k}|_{H_{Z_k}}\le \phi^*(F|_{H_{Y'}})$, it follows that  
	  $\delta_{H_{Z_k}\rightarrow H_0}(F_{Z_k}|_{H_{Z_k}})$ is bounded by an integer which depends only on $n$. 
	 Now since $Z_j\dashrightarrow Z_{j+1}$ are all flips for $j>1$, $\delta_{H_{Z_j}\rightarrow H_0}(F_{Z_j}|_{H_{Z_j}})$ is bounded by Lemma \ref{del}, for all $j>1$.
	 \begin{description}	
	 \item[Step 5] If $Z_1\dashrightarrow Z_2$ is a flop, then $\delta_{H_{Z_1}\rightarrow H_0}(E_{Z_1}|_{H_{Z_1}})$ is bounded by a number which depends only on $n$.
	 \end{description}
	 By Step 1 we know that $\delta_{H_{Z_2}\rightarrow H_0}(E_{Z_2}|_{H_{Z_2}})$ is bounded. Thus, by Step 2, we may assume that there exists a flopped curve 
	$C_{Z_2}\subset Z_2$ which is contained in $H_{Z_2}$.
	 Then since $\delta_{H_{Z_2}\rightarrow H_0}(E_{Z_2}|_{H_{Z_2}})$ is bounded, it follows that $E_{Z_2}\cdot C_{Z_2}$ is bounded. Step 4 implies that $F_{Z_2}\cdot C_{Z_2}$ is also bounded. Let $Z_1\to V$ and $Z_2\to V$ be the flopping contractions. 
	 Then $E_{Z_2}\equiv_V\lambda F_{Z_2}$ for some rational number $\lambda$ which is bounded. It follows that $E_{Z_1}\equiv_V\lambda F_{Z_1}$.\par
	 Now let $C_{Z_1}$ be a flopping curve of $Z_1\dashrightarrow Z_2$ and let $C_Y$ be the image of the $C_{Z_1}$ on $Y$. Then
	 \[ K_Y\cdot C_Y+a(F_{Z_1},Y)F_{Z_1}\cdot C_{Z_1}=K_{Z_1}\cdot C_{Z_1}=0.\] 
	  By \cite[Theorem 0]{b}, we know that $0>K_Y\cdot C_Y>-1$. Moreover, $a(F_{Z_1},Y)=1/r$ where $r$ is the Cartier index of $K_Y$ near a singular point of $Y$. It follows that $0<F_{Z_1}\cdot C_{Z_1}<r\leq 2n$ by Lemma \ref{bka}. Since $E_{Z_1}\equiv_V\lambda F_{Z_1}$ for some bounded $\lambda$ and since the Cartier index of $F_{Z_1}$ is bounded, it follows  that $E_{Z_1}\cdot C_{Z_1}$ is bounded by an integer which depends only on $n$. Using the same argument as in the proof of Lemma \ref{del} (3) (the flopping locus of $Z_1\dashrightarrow Z_2$ is irreducible by \cite[Lemma 3.11]{c}), it follows that $\delta_{H_{Z_1}\rightarrow H_0}(E_{Z_1}|_{H_{Z_1}})$ is bounded.\par
	  \medskip
	  
	 Finally, Step 1 and Step 5 imply the claim.
\end{proof}

\begin{lem}\label{kcb} Fix a positive integer $n$.
	Let $X\dashrightarrow X'$ be a three-dimensional terminal flip over $W$ such that $dep(X)=n$. 
	
	Then  for any flipped curve $C'\subset X'$, we have that $K_{X'}\cdot C'$ is bounded by an integer which depends only on $n$.
\end{lem}
\begin{proof}
	If there is more than one flipping curve on $X$, then we can run an analytic $K_X$-MMP over $W$. 
	Indeed, we can decompose the map $X\dashrightarrow X'$ into a  sequence of analytic flips 
	\[X=X_1\dashrightarrow X_2\dashrightarrow ...\dashrightarrow X_{k-1}\dashrightarrow X_k=X'\]
such that the flipping locus of $X_i\dashrightarrow X_{i+1}$ is irreducible for all $i$ and any flipped curve of $X\dashrightarrow X'$ is the proper transform of a flipped curve of $X_i\dashrightarrow X_{i+1}$ on $X'$ for some $i$. By Corollary \ref{ntl2}, we only need to prove that the statement holds for the flip $X_i\dashrightarrow X_{i+1}$ for all $i$. Thus, we may assume that the flipping locus of $X\dashrightarrow X'$ is irreducible. 
In this case, a general member $H_W\in |-K_W|$ has Du Val singularities by \cite[Theorem 2.2]{km}. Moreover, the singularities of $H_W$ depend only on the Cartier indices and the axial multiplicities of the singular points on $X$. By Lemma \ref{bka}, it follows that $\bar{\rho}(H_W)$ is bounded by an integer $N(n)$ which depends only on $n$. One can also assume that $N(n)\geq n$.\par
	Consider the diagram\[\vc{\xymatrix@C=0.8cm{Y_1\ar[d] \ar@{-->}[r] & ... \ar@{-->}[r] & Y_k\ar[d]\\ X\ar[rd] & & X'\ar[ld] \\ & W & }}\] as in Theorem \ref{chf}. By Remark \ref{flpc} we know that $dep(Y_i)<n$ for all $i$. By induction on $n$, we may assume that if $Y_i\dashrightarrow Y_{i+1}$ is a flip, then for any flipped curve $C_{i+1}\subset Y_{i+1}$ we have that $K_{Y_{i+1}}\cdot C_{i+1}$ is bounded by an integer depending only on $n$. Let $E=exc(Y_k\rightarrow X')$. 
	
	\medskip 
	
We distinguish two cases:
	\begin{description}
	\item[Case 1] $C'$ is not contained in the centre of $E$ on $X'$.
	\end{description}\par
	Let $C_{Y_k}$ be the proper transform of $C'$ on $Y_k$. Then \[K_{Y_k}\cdot C_{Y_k}=K_{X'}\cdot C'+a(E,X')E\cdot C_{Y_k}\geq K_{X'}\cdot C'>0.\] This means that $C_{Y_k}$ is the proper transform of a flipped curve of $Y_i\dashrightarrow Y_{i+1}$ on $Y_k$ for some $i$. By the induction hypothesis and Corollary \ref{ntl2}, we know that $K_{Y_k}\cdot C_{Y_k}$ is bounded by an integer depending only on $n$, hence so is $K_{X'}\cdot C'$.
	\begin{description}
	\item[Case 2] $Y_k\rightarrow X'$ is a divisorial contraction to $C'$.
	\end{description}\par
	Notice that $C'$ is a smooth curve by \cite[Corollary 3.3]{c} and if  $H_{Y_k}$ is the proper transform of $H_W$ on $Y_k$ then 
	$H_{Y_k}\rightarrow H_W$ is a partial resolution by \cite[Lemma 3.4]{c}. Recall that $N(n)$ is an integer which is greater than both $n$ and $\bar{\rho}(H_W)$. By Lemma \ref{ecb} we know that $\delta_{H_{Y_k}\rightarrow H_W}(E|_{H_{Y_k}})$ is bounded by an integer which depends only on $N(n)$. Notice that $Y_k\rightarrow X'$ is generically a blow-up along $C'$. Since $H_{Y_k}\rightarrow H_W$ is a partial resolution, if we denote by $H'$ the strict transform of $H_W$ on $X'$ then  $\mmul_{C'}H'=1$ and, in particular,  there is exactly one component of $E\cap H_{Y_k}$ which maps surjectively to $C'$. Let $C_{Y_k}$ be this component, then $|E\cdot C_{Y_k}|$ is bounded by an integer which depends only on $N(n)$ since $\delta_{H_{Y_k}\rightarrow H_W}(E|_{H_{Y_k}})$ is bounded.\par
	Now $K_{Y_k}\cdot C_{Y_k}=K_{X'}\cdot C'+E\cdot C_{Y_k}$. If $K_{Y_k}\cdot C_{Y_k}\leq 0$, then $K_{X'}\cdot C'\leq -E\cdot C_{Y_k}$ is bounded. If $K_{Y_k}\cdot C_{Y_k}>0$, then $C_{Y_k}$ is the proper transform of a flipped curve of $Y_i\dashrightarrow Y_{i+1}$ for some $i$. It follows that $K_{Y_k}\cdot C_{Y_k}$ is bounded by the induction hypothesis and by Corollary \ref{ntl2}. This implies that $K_{X'}\cdot C'$ is bounded.
\end{proof}

\begin{lem}\label{e3b}
Let $f\colon Y\rightarrow X$ be a divisorial contraction between projective terminal threefolds which contracts a divisor $E$ to a smooth curve
	$C$. 
	
	Then $|K_Y^3-K_X^3|$ is bounded by an integer which depends only on $g(C)$, $K_X\cdot C$ and the basket data of $Y$.
\end{lem}
\begin{proof}
	Since $Y\rightarrow X$ is generically a blow-up along $C$, we know that $a(E,X)=1$. Hence \[K_Y^3-K_X^3=3E^2\cdot f\st K_X+E^3=-3K_X\cdot C+E^3.\] Thus, we only need to bound $|E^3|$.\par
	Consider the following two exact sequences \[ 0\rightarrow\Oo_Y(-E)\rightarrow\Oo_Y\rightarrow\Oo_E\rightarrow0\] and
	\[ 0\rightarrow\Oo_Y(K_Y-E)\rightarrow\Oo_Y(K_Y)\rightarrow\Oo_E(K_Y)\rightarrow0.\]
The singular Riemann-Roch formula (cf. Section \ref{srr}) yields 
\[\chi(\Oo_E)=\chi(\Oo_Y)-\chi(\Oo_Y(-E))=
\frac{1}{12}(-5K_X\cdot C+6E^3)+\frac{1}{12}E\cdot c_2(Y)+\theta_1\]
and 
\[\chi(\Oo_E(K_Y))=\chi(\Oo_Y(K_Y))-\chi(\Oo_Y(K_Y-E))=
\frac{1}{12}K_X\cdot C+\frac{1}{12}E\cdot c_2(Y)+\theta_2\] where $\theta_1$ and $\theta_2$ are constants which depend only on the basket data of $Y$ and the integer $i_P$ such that $-E\sim i_PK_Y$ at any  singular point $P\in Y$.
 On the other hand, 
by the  Kawamata-Viehweg vanishing theorem (cf. \cite[Theorem 1-2-3]{kmm}) we have that $R^kf\ts(\Oo_Y(iK_Y-jE))=0$ if $k\geq 1$ and $i-1-j\leq0$.
By considering the exact sequence \[0\rightarrow\Oo_Y(-E)\rightarrow\Oo_Y\rightarrow\Oo_E\rightarrow0\] we know that $R^kf\ts(\Oo_E)=0$ for $k\geq 1$. Similarly, by considering the exact sequence \[0\rightarrow\Oo_Y(K_Y-E)\rightarrow\Oo_Y(K_Y)\rightarrow\Oo_E(K_Y)\rightarrow0\] we know that $R^kf\ts(\Oo_E(K_Y))=0$ for $k\geq 1$. Thus, $\chi(\Oo_E)=\chi(f\ts\Oo_E)=\chi(\Oo_C)$ and $\chi(\Oo_E(K_Y))=\chi(f\ts(\Oo_E(K_Y)))$. Note that the push-forward of the second exact sequence is 
	\[0\rightarrow \Oo_X(K_X)\rightarrow\Oo_X(K_X)\otimes f\ts\Oo_Y(E)=\Oo_X(K_X)\rightarrow f\ts(\Oo_E(K_Y))\rightarrow0,\]
which implies $\chi(\Oo_E(K_Y))=\chi(f\ts(\Oo_E(K_Y)))=0$, or $\frac{1}{12}E\cdot c_2(Y)=-\frac{1}{12}K_X\cdot C-\theta_2$. Thus,
	\[ \chi(\Oo_C)=\chi(\Oo_E)=\frac{1}{2}(-K_X\cdot C+E^3)+\theta_1-\theta_2,\] or \[ E^3=K_X\cdot C+2\chi(\Oo_C)+2(\theta_2-\theta_1).\]
Since $\theta_1$ and $\theta_2$ take only finitely many possible value if the basket data of $Y$ is given, $|E^3|$ is bounded by an integer which depends only on $g(C)$, $K_X\cdot C$ the basket data of $Y$.
\end{proof}
\begin{lem}\label{bdp}
	Fix a positive integer $n$. Let $Y\rightarrow X$ be a divisorial contraction to a point between terminal threefolds such that $dep(Y)=n$. 
	
	Then $|K_Y^3-K_X^3|$ is bounded by an integer which depends only on $n$. 
\end{lem}
\begin{proof}
	Let $E$ be the exceptional divisor. Then $K_Y^3-K_X^3=a^3E^3$ where $a=a(E,X)$. By \cite[Table 1, Table 2]{k} we know that $0<aE^3\leq 4$. Let $r$ be the Cartier index of $Y$, then $E^3\ge \frac{1}{r^3}$, hence $a\le 4r^3$. Thus, $0<a^3E^3\leq 64r^6$. Since $r$ is bounded by an integer which depends only on $n$ by Lemma \ref{bka}, we have that  $|a^3E^3|=|K_Y^3-K_X^3|$ is bounded by an integer which depends only on $n$.
\end{proof}

\begin{pro}\label{bk3}
	Fix a positive integer $n$. Let $X\dashrightarrow X'$ be a three-dimensional terminal flip such that $dep(X)=n$. 
	
	Then $|K_X^3-K_{X'}^3|$ is bounded by an integer which depends only on $n$
\end{pro}
\begin{proof}
	By Theorem \ref{chf}, we have a  diagram \[\vc{\xymatrix@C=0.8cm{Y_1\ar[d] \ar@{-->}[r] & ... \ar@{-->}[r] & Y_k\ar[d]\\ X\ar[rd] & & X'\ar[ld] \\ & W & }}\]  We know that $dep(Y_i)<n$ for all $i$. If $Y_i\dashrightarrow Y_{i+1}$ is a flip, then by induction on $n$, we may assume that $|K_{Y_i}^3-K_{Y_{i+1}}^3|$ is bounded by an integer depending on $n$. Also if $Y_i\dashrightarrow Y_{i+1}$ is a flop, then $K_{Y_i}^3=K_{Y_{i+1}}^3$. Since $k\leq n+1$, we know that $|K_{Y_1}^3-K_{Y_k}^3|$ is bounded by an integer which depends only on $n$.\par
	We know that $Y_1\rightarrow X$ is a $w$-morphism, so $|K_{Y_1}^3-K_X^3|$ is bounded by an integer which depends only on $n$ by Lemma \ref{bdp}. Similarly, if $Y_k\rightarrow X'$ is a divisorial contraction to a point, then we can also bound $|K_{Y_k}^3-K_{X'}^3|$. Assume that $Y_k\rightarrow X'$ is a divisorial contraction to a curve $C'$, then $C'$ is a smooth rational curve. Since $dep(Y_k)<n$, the basket data of $Y_k$ is bounded by Lemma \ref{bka}. Hence $|K_{Y_k}^3-K_{X'}^3|$ is bounded by an integer which depends only on $n$ by Lemma \ref{kcb} and Lemma \ref{e3b}.\par
	Finally since $|K_{Y_1}^3-K_X^3|$, $|K_{Y_1}^3-K_{Y_k}^3|$ and $|K_{Y_k}^3-K_{X'}^3|$ are all bounded, it follows that $|K_X^3-K_{X'}^3|$ is bounded by an integer which depends only on $n$.
\end{proof}

\begin{proof}(Proof of Theorem \ref{thm1})
	By Lemma \ref{depr} we know that $dep(X_i)<\rho(X)\leq b_2(X)$.  Proposition \ref{bk3} implies the Theorem.
\end{proof}

\section{Cubic form after a flip}

\subsection{Some useful lemmas}

In this subsection we assume that $\phi:X\rightarrow W$ is a birational morphism between normal projective varieties and $Z=exc(\phi)$.

\begin{lem}\label{der}
	There exists an analytic open set $U_W\subset W$ which can deformation retract to $\phi(Z)$, such that $U=\phi^{-1}(U_W)$ can deformation retract to $Z$.
\end{lem}
\begin{proof}
	First choose an arbitrary open set $U_0\subset X$ which can deformation retract to $Z$. Notice that the interior of $\phi(U_0)$ is non-empty since $\phi(U_0-Z)\cong U_0-Z$ is non-empty and is contained in $\phi(U_0)$. Let $U_W\subset\phi(U_0)$ be an open set which can deformation retract to $\phi(Z)$. Now $U=\phi^{-1}(U_W)$ is contained in $U_0$ since $\phi$ is an isomorphism outside $Z$, and $U_0$ already contains $Z$. Thus $U$ can deformation retract to $Z$.
\end{proof}
\begin{lem}\label{hiso}$ $
	\begin{enumerate}[(1)]
	\item If $H_{i-1}(Z,\Z)=H_i(\phi(Z),\Z)=0$, then the push-forward map $H_i(X,\Z)\rightarrow H_i(W,\Z)$ is surjective.
	\item $H^i(X,\Z)\cong H^i(W,\Z)$ if $i>2\dim Z+1$.
	\end{enumerate}
\end{lem}
\begin{proof}
	Let $U\subset X$ be the open set in Lemma \ref{der} which can deformation retract to $Z$. Let $V\subset X$ be another open set which does not intersect $Z$ and contains $X-U$. Let $U_W\subset W$ be the open set in Lemma \ref{der} and let $V_W=\phi(V)\subset W$. Then $U_W$ and $V_W$ cover $W$. Notice that $V_W\cong V$ and $U\cap V\cong U_W\cap V_W$.\par
	Now consider the following two Mayer-Vietoris sequences
	\[\vc{\xymatrix@C=0.3cm{ H_i(U,\Z)\oplus H_i(V,\Z)\ar[r]\ar[d]_{\alpha} & H_i(X,\Z)\ar[d]_{\beta}\ar[r] & H_{i-1}(U\cap V,\Z)\ar[d]_{\gamma}\ar[r] & H_{i-1}(U,\Z)\oplus H_{i-1}(V,\Z)\ar[d]_{\delta}\\
	H_i(U_W,\Z)\oplus H_i(V_W,\Z)\ar[r] & H_i(W,\Z)\ar[r] & H_{i-1}(U_W\cap V_W,\Z) \ar[r] & H_{i-1}(U_W,\Z)\oplus H_{i-1}(V_W,\Z)} }.\]
	Since $H_i(U_W,\Z)=H_i(\phi(Z),\Z)=0$ and $H_{i-1}(U,\Z)=H_{i-1}(Z,\Z)=0$, we know that $\alpha$ is surjective and $\delta$ is injective. Also $\gamma$ is an isomorphism and hence $\beta$ is surjective by the four lemma.\par
	To prove (2) consider the following two Mayer-Vietoris sequences for cohomology
	\[\vc{\xymatrix@C=0.3cm{ H^{i-1}(U_W,\Z)\oplus H^{i-1}(V_W,\Z)\ar[r]\ar[d] & H^{i-1}(U_W\cap V_W,\Z)\ar[d]\ar[r] & H^i(W,\Z)\ar[d]\ar[r] & H^i(U_W,\Z)\oplus H^i(V_W,\Z) \ar[r]\ar[d] & H^i(U_W\cap V_W,\Z)\ar[d] \\
	H^{i-1}(U,\Z)\oplus H^{i-1}(V,\Z)\ar[r] & H^{i-1}(V\cap U,\Z)\ar[r] & H^i(X,\Z)\ar[r] & H^i(U,\Z)\oplus H^i(V,\Z)\ar[r] & H^i(V\cap U,\Z)} }.\]
	Since $H^j(U,\Z)=H^j(Z,\Z)=0$ and $H^j(U_W,\Z)=H^j(\phi(Z),\Z)=0$ for $j=i-1$ and $i$, we know that $H^i(X,\Z)\cong H^i(W,\Z)$ by the five lemma.
\end{proof}
\begin{lem}\label{hbasis}
	Let $k=2\dim Z$. Assume that $b_k(X)=b_k(W)+1$ and $H_{k-1}(Z,\Z)=0$, then \[H^k(X,\Z)=\Z\cdot \eta_1\oplus T\] for some $\eta_1\in H^k(X,\Z)$, where $T$ is the submodule of $H^k(X,\Z)$ generated by $\phi\st H^k(W,\Z)$ and $H^k(X,\Z)_{torsion}$.
\end{lem}
\begin{proof}
	Let $\xi_1\in H_k(X,\Z)$ be a cycle which is supported on $Z$ and is a part of a basis of $H_k(X,\Z)$. By the surjectivity part of the universal coefficient theorem, we know that there exists $\eta_1\in H^k(X,\Z)$ such that $\eta_1(\xi_1)=1$. Given any element $\eta\in H^k(X,\Z)$, let $\eta'=\eta-\eta(\xi)\eta_1$. We only need to show that $\eta'\in T$. Notice that $\eta'(\xi_1)=0$ which implies that $\eta'(\xi')=0$ for all $\xi'\in H_k(X,\Z)$ supported on $Z$. Indeed, since $b_k(X)=b_k(W)+1$, we know that $\eta'$, when viewed as an element in $H^k(X,\Q)$, can be written as $\lambda[A^{\dim Z}]+\phi_i\st\tau$ for some $\lambda\in\Q$, for some ample divisor $A$ on $X$ and for some $\tau\in H^k(W,\Q)$. Now $\eta'(\xi_1)=0$ implies that $\lambda=0$, hence $\eta'(\xi')=0$ for any $\xi'$ which is supported on $Z$.\par
	Let $\xi_{2,W}$, ..., $\xi_{n,W}\in H_k(W,\Z)$ be a basis. Notice that $H_k(\phi(Z),\Z)=0$ since $2\dim\phi(Z)<2\dim Z=k$. By Lemma \ref{hiso} we know that $H_k(X,\Z)\rightarrow H_k(W,\Z)$ is surjective, hence there exist $\xi_2$, ..., $\xi_n$ such that $\phi\ts(\xi_j)=\xi_{j,W}$ for $j=2$, ..., $n$. Again by the universal coefficient theorem, there exists $\tau\in H^k(W,\Z)$ such that $\tau(\xi_{j,W})=\eta'(\xi_j)$. Let $\eta''=\eta'-\phi\st\tau$, then $\eta''(\xi_j)=0$ for $j=2$, ..., $n$ and also $\eta''(\xi')=0$ for all $\xi'\in H_k(X,\Z)$ which is supported on $Z$. Therefore $\eta''$ intersects every element in $H_k(X,\Z)$ trivially. The universal coefficient theorem implies that $\eta''$ is torsion, hence $\eta'\in T$. 
\end{proof}
\subsection{The topology of flips}
\begin{lem}\label{l_b2}
	Assume that $X\rightarrow W$ is a birational morphism between varieties with rational singularities such that $\rho(X/W)=1$. 
	
	Then $b_2(X)=b_2(W)+1$. 
\end{lem}
\begin{proof}
	The proof of \cite[Lemma 2.16 (3)]{ct} works unchanged under our assumption.
\end{proof}

\begin{lem}\label{h4h2}
	Assume that $X$ is a terminal threefold. 
	
	Then the natural map $H^4(X,\Z)_{free}\rightarrow H_2(X,\Z)_{free}$ is surjective.
\end{lem}
\begin{proof}
	Let $\tl{X}\rightarrow X$ be a resolution of singularities such that the exceptional locus $E$ is normal crossing. From \cite[Lemma 12.1.1]{km} we know that $H^1(E,\Z)=0$, hence $H_1(E,\Z)=0$ since there exists a surjective map $H^1(E,\Z)\rightarrow Hom(H_1(E,\Z),\Z)$ by the universal coefficient theorem. Therefore the push-forward map $H_2(\tl{X},\Z)\rightarrow H_2(X,\Z)$ is surjective by Lemma \ref{hiso}.\par
	From \cite[Theorem 2]{cj} we may assume that the morphism $\tl{X}\rightarrow X$ is a sequence of divisorial contractions between terminal threefolds $\tl{X}=X_0\rightarrow X_1\rightarrow ... \rightarrow X_k=X$. Let $E_i\in H^4(X,\Z)$ be the cycle corresponding to the exceptional locus of $X_i\rightarrow X_{i+1}$ for $i=0,\dots,k-1$, and let $\sigma_i\in H^4(X_i,\Z)$ be a class such that $\sigma_i(E_i)=1$ and $\sigma_i(Z)=0$ for any $Z\in H_4(X_i,\Z)$ whose support does not contain $E_i$. Let $\tl{\sigma}_i$ be the pull-back of $\sigma_i$ on $\tl{X}$. Then $H^4(\tl{X},\Z)_{free}$ is generated by $\tl{\sigma}_0$, ..., $\tl{\sigma}_{k-1}$ and the pull-back of $H^4(X,\Z)_{free}$. Let $\xi_i\in H_2(\tl{X},\Z)$ be the cycle corresponding to $\tl{\sigma}_i$ for $i=0$, ..., $k-1$. Then the push-forward of $\xi_i$ on $X$ is zero.\par
	Given a class $\zeta\in H_2(X,\Z)_{free}$, there exists $\tl{\zeta}\in H_2(\tl{X},\Z)_{free}$ which maps to $\zeta$. By  Poincar\'{e} duality, we have that $\tl{\zeta}$ corresponds to an element $\tl{\tau}\in H^4(\tl{X},\Z)_{free}$. 
	Thus, there exists $\tl{\sigma}\in H^4(\tl{X},\Z)_{free}$ which is generated by $\tl{\sigma}_0$, ..., $\tl{\sigma}_{k-1}$ and such that  $\tl{\tau}+\tl{\sigma}$ is the pull-back of a cocycle $\tau\in H^4(X,\Z)_{free}$. Let $\tl{\xi}\in H_2(\tl{X},\Z)_{free}$  be such that $\tl{\xi}$ corresponds to $\tl{\sigma}$, then the push-forward of $\tl{\zeta}-\tl{\xi}$ is also $\zeta$. This means that $\tau$ is an element in $H^4(X,\Z)_{free}$ which maps to $\zeta$.
\end{proof}

\begin{lem}\label{ch4h2}
	Assume that $X$ is a terminal threefold and $X\rightarrow W$ is a flipping or a flopping contraction. 
	
	Then the natural map $H^4(W,\Z)_{free}\rightarrow H_2(W,\Z)_{free}$ is surjective.
\end{lem}
\begin{proof}
	If $X\rightarrow W$ is a flopping contraction, then $W$ is also terminal, and the statement follows from Lemma \ref{h4h2} directly.\par
	
	Now assume that $X\rightarrow W$ is a flipping contraction. Let $Z$ be the exceptional locus of $X\rightarrow W$. Lemma \ref{hiso} implies that $H_2(X,\Z)_{free}\rightarrow H_2(W,\Z)_{free}$ is surjective and $H^4(W,\Z)_{free}\cong H^4(X,\Z)_{free}$ since $H_1(Z,\Z)=0$ for $Z$ is a tree of rational curves by \cite[Corollary 1.3]{m2}. By Lemma \ref{h4h2} we know that $H^4(X,\Z)_{free}\rightarrow H_2(X,\Z)_{free}$ is surjective, hence the following commutative diagram
	\[\vc{\xymatrix{H^4(X,\Z)_{free}\ar[r] & H_2(X,\Z)_{free}\ar[d]  \\ H^4(W,\Z)_{free}\ar[u]\ar[r]& H_2(W,\Z)_{free}}}\]
	implies that $H^4(W,\Z)_{free}\rightarrow H_2(W,\Z)_{free}$ is surjective.
\end{proof}

\begin{lem}
	Assume that $X$ is a terminal threefold. 
	
	Then $b_2(X)=b_4(X)$.
\end{lem}
\begin{proof}
	Let $\tl{X}\rightarrow X$ be a feasible resolution of $X$ (cf. \cite{cj}). Then $\tl{X}\rightarrow X$ is a composition of divisorial contractions between terminal threefolds. By  Poincar\'{e} duality we have that $b_2(\tl{X})=b_4(\tl{X})$. We also have  that
	\[b_2(\tl{X})-b_2(X)=b_4(\tl{X})-b_4(X)=\rho(\tl{X}/X)\] by \cite[Lemma 2.16]{ct}. Hence $b_2(X)=b_4(X)$.
\end{proof}

\begin{pro}\label{cbf}
	Let $n$ be an integer and let $F\in\Z[x_1,...,x_k]$ be a  cubic form.  Assume that	
	\begin{enumerate}[(1)]
	\item $X\dashrightarrow X'$ is a three-dimensional terminal flip over $W$ with $dep(X)\leq n$, and
	\item there exists a basis $\eta_1$, ..., $\eta_k$ of $H^2(X,\Z)_{free}$ such that $\eta_2$, ..., $\eta_k$ are pull-backs of elements in $H^2(W,\Z)_{free}$, and $F$ is the associated cubic form with respect to this basis.
	\end{enumerate}

	Then there exist  finitely many cubic forms $F'_1\dots, F'_m$ depending only on $F$ and $n$  such that 
	the cubic form  $F_{X'}$ of $X'$ is equivalent to  $F'_i$ for some $i=1,\dots,m$.
\end{pro}
\begin{proof}
	By Lemma \ref{ch4h2} we know that the natural map $H^4(W,\Z)_{free}\rightarrow H_2(W,\Z)_{free}$ is surjective. Lemma  \ref{l_b2} implies that $b_2(W)=b_2(X)-1$. Since  $b_4(W)=b_4(X)=b_2(X)$, it follows that the above map has a kernel of rank one. Let $\tau_1\in H^4(W,\Z)_{free}$ be the generator of the kernel. Also fix a generating set $\zeta_2$, ..., $\zeta_k$ of $H_2(W,\Z)_{free}$. For each $\zeta_i$, we fix a lifting $\tau_i\in H^4(W,\Z)_{free}$ of $\zeta_i$. Let $\xi_i\in H_2(X,\Z)_{free}$ be the cycle corresponding to the pull-back of $\tau_i$ in $H^4(X,\Z)_{free}$ for $i=1$, ..., $k$. Then $\xi_1$ is supported on the flipping locus and the push-forward of $\xi_i$ is $\zeta_i$ for $i>1$. Since the natural map
	\[ H^4(W,\Z)_{free}\cong H^4(X,\Z)_{free}\rightarrow H_2(X,\Z)_{free}\] is surjective by Lemma \ref{h4h2}, we know that $\xi_1$, ..., $\xi_k$ generates $H_2(X,\Z)_{free}$. We know that $\eta_i(\xi_1)=0$ for $i>1$. By the surjectivity part of the universal coefficient theorem, there exists an element $\eta\in H^2(X,\Z)_{free}$ such that $\eta(\xi_1)=1$. This means that $\eta_1(\xi_1)=1$. For all $i>1$, by replacing $\xi_i$ by $\xi_i-\eta_1(\xi_i)\xi_1$, we may assume that $\eta_1(\xi_i)=0$.\par
	We can choose $\xi'_1$, ..., $\xi'_k\in H_2(X',\Z)_{free}$ such that $\xi_i$ and $\xi'_i$ corresponds to the same pull-back of an element in $H^4(W,\Z)_{free}$. Then $\xi'_1$ is supported on the flipped locus, the push-forward of $\xi'_i$ is exactly $\zeta_i$ for $i>1$, and $\xi'_1$, ..., $\xi'_k$  generate $H_2(X',\Z)_{free}$. We choose $\eta'_1\in H^2(X',\Z)_{free}$ be the element such that $\eta'_1(\xi'_1)=1$ and $\eta'_1(\xi'_i)=0$ for $i>1$. For each $2\leq i\leq n$, let $\eta'_i\in H^2(X',\Z)_{free}$ be the element such that $\eta_i$ and $\eta'_i$ correspond to the same pull-back of an element in $H^2(W,\Z)_{free}$. Then $\eta'_1,\dots,\eta'_n$ form a basis of $H^2(X',\Z)_{free}$.\par
	Let $\sigma\coloneqq \eta_1$. Since $b_2(X)=b_2(W)+1$ and $K_X$ is ample over $W$, we can write $\sigma=\lambda K_X+\tau$ where $\lambda\in \mathbb Q$ and  $\tau$ is the pull-back of some element in $H^2(W,\Q)$. Let $\tau'\in H^2(X',\Z)_{free}$ be the corresponding element of $\tau$ and define $\sigma'\coloneqq \lambda K_{X'}+\tau'$. Then by the projection formula it follows that $\sigma^j\cdot\nu={\sigma'}^j\cdot\nu'$ if $\nu$ and $\nu'$ corresponds to the same pull-back of an element in $H^{6-2j}(W,\Z)$, for $j=1$, $2$. Thus, $\sigma'(\xi'_1)\neq 0$ and $\sigma'(\xi'_i)=0$ for $i>1$ and, in particular, $\sigma'=\mu \eta'_1$ where $\mu=\frac{1}{\lambda}K_{X'}\cdot C'$.\par
	Now one can see that $\eta_1^3-\mu^3{\eta'}_1^3=\lambda^3(K_{X'}^3-K_X^3)$, $\eta_1^2\cdot\eta_i=\mu^2{\eta_1'}^2\cdot\eta'_i$ for $i>1$, and $\eta_1\cdot\eta_i\cdot\eta_j=\mu\eta_1'\cdot\eta_i'\cdot\eta_j'$. Therefore if we write 
	\[
	F_X(x_1,...,x_k)=ax_1^3+x^2_1B(x_2,...,x_k)+x_1C(x_2,...,x_k)+D(x_2,...,x_k)
	\]
	for some homogeneous polynomials $B,C,D\in \Z[x_1,\dots,x_k]$ of degree $1,2$ and $3$ respectively, then 
	\begin{align*}
	F_{X'}(x_1,...,x_k)=\frac{1}{\mu^3}(a+\lambda^3&(K_{X'}^3-K_X^3))x_1^3+\frac{1}{\mu^2}x_1^2B(x_2,...,x_k)\\
		&+\frac{1}{\mu}x_1C(x_2,...,x_k)+D(x_2,...,x_k).
	\end{align*}
	Let $r$ and $r'$ be the Cartier index of $K_X$ and $K_{X'}$ respectively. Then,  by Lemma \ref{bka} and Proposition \ref{bk3},  we have that  $r, r'$ and $r^3{r'}^3(K_{X'}^3-K_X^3)$ are integers which are bounded by a constant depending only on $n$. It follows that $K_{X'}^3-K_X^3$ admits only finitely many possible values. Moreover, we know that $\lambda=\frac{1}{K_X\cdot C}$. Since $K_X\cdot C\in \frac{1}{r}\Z\cap(-1,0)$ by \cite[Theorem 0]{b}, we have that $\lambda$ admits only finitely many possible values. Finally $\mu=\frac{1}{\lambda}K_{X'}\cdot C'$. By Lemma \ref{kcb} we know that $K_{X'}\cdot C'$ is bounded. Since $r'K_{X'}\cdot C'$ is an integer, it follows that $\mu$ also admits only finitely possible values. Thus, the proposition follows.
\end{proof}

\begin{proof} (Proof of Theorem \ref{thm2})
	By Lemma \ref{depr} we know that $dep(X_i)<\rho(X)\leq b_2(X)$. The assumption (2) of Proposition \ref{cbf} holds because of the fact that $H_1(exc(X_i\rightarrow W_i),\Z)=0$ (cf. proof of Lemma \ref{ch4h2}) and Lemma \ref{hbasis}. Now Proposition \ref{cbf} implies our theorem.
\end{proof}

\end{document}